\documentclass[12pt]{amsart}
\usepackage{amssymb}
\usepackage{amscd}
\usepackage{amsfonts}
\usepackage[all]{xy}
\usepackage{eucal}
\usepackage[T1]{fontenc}

\newcommand{\Mo}{(M,\omega )}
\newcommand{\Wo}{(W,\omega_W )}

\newcommand\Symp{\operatorname{Symp}}
\newcommand\Ham{\operatorname{Ham}}

\newtheorem{theorem}[subsection]{Theorem}
\newtheorem{proposition}[subsection]{Proposition}
\newtheorem{lemma}[subsection]{Lemma}
\newtheorem{cor}[subsection]{Corollary}

\theoremstyle{definition}

\newtheorem{example}[subsection]{Example}
\theoremstyle{remark}
\newtheorem{remark}[subsection]{Remark}

\title{On nondegenerate coupling forms}
\author{Jarek K\c edra}
\address{JK: University of Aberdeen, University of Szczecin}
\email{kedra@abdn.ac.uk}
\author{Aleksy Tralle}
\address{AT: University of Warmia and Mazury in Olsztyn}
\email{tralle@matman.uwm.edu.pl}
\author{Artur Woike}
\address{AW: University of Warmia and Mazury in Olsztyn}
\email{awoike@matman.uwm.edu.pl}

\begin{document}

\begin{abstract}
The aim of the present paper is to investigate new classes of
symplectically fat fibre bundles.  We prove a general existence
theorem for fat vectors with respect to the canonical invariant
connections.  Based on this result we give new proofs of some
constructions of symplectic structures. This includes twistor bundles
and locally homogeneous complex manifolds. The proofs are conceptually simpler and
allow for obtaining more general results.
\end{abstract}


\maketitle

\section{Introduction}\label{S:intro}

\subsection{Fat vectors}\label{SS:vectors}
Let  $G\rightarrow P\rightarrow B$ 
be a principal bundle with a connection.
Let $\theta$ and $\Theta$ be the connection one-form and
the curvature form of the connection, respectively.  Both forms have
values in the Lie algebra $\mathfrak{g}$ of the group $G$. 
Denote the pairing between $\mathfrak{g}$ and
its dual $\mathfrak{g}^*$ by $\langle\,,\rangle$.  By definition, a vector
$u\in\mathfrak{g}^*$ is {\bf {\em  fat}}, if the two--form
$$
(X,Y)\rightarrow \langle\Theta(X,Y),u\rangle
$$
is nondegenerate for all {\it horizontal} vectors $X,Y$. Note
that if a connection admits at least one fat vector then it admits the
whole coadjoint orbit of fat vectors. 

Let $\Mo$ be a symplectic manifold with a Hamiltonian action
of a group $G$ and the moment map $\Psi:M\to \mathfrak g^*$. 
Consider the associated Hamiltonian bundle
$$
\Mo \to E:=P\times_G M \to B.
$$

Sternberg \cite{MR0458486} constructed a certain closed two--form
$\Omega\in\Omega^2(E)$ associated with the connection $\theta$.  It is
called the {\bf {\em coupling form}} and pulls back to the symplectic form
on each fibre and it is degenerate in general. However, if the image
of the moment map consists of fat vectors then the coupling form is
nondegenerate, hence symplectic.  This was observed by Weinstein in
\cite[Theorem 3.2]{MR82a:53038} where he used this idea to give a new
construction of symplectic manifolds.  In the sequel, the bundles with
a nondegenerate coupling form will be called {\bf {\em symplectically fat.}}
Let us state the result of Sternberg and Weinstein precisely.

\begin{theorem}[Sternebrg-Weinstein] 
Let $G\rightarrow P\rightarrow B$ be a principal bundle. Let $\Mo$ be a
a symplectic manifold with a Hamiltonian $G$-action and the moment map
$\Psi: M\rightarrow\mathfrak{g}^*$.  If there exists a connection in the
principal bundle such that all vectors in $\Psi(M)\subset\mathfrak{g}^*$
are fat, then the coupling form on the total space of the associated
bundle
$$
F\rightarrow P\times_G F\rightarrow B
$$
is symplectic.
\end{theorem}

Notice that, if the base is symplectic then, according to Thurston
\cite{MR53:6578}, the existence of a coupling form suffices to construct a
fibrewise symplectic form.  Thus the most important and interesting
examples provided by fat bundles are the ones with nonsymplectic or
highly connected bases, e.g. spheres.

Another interesting feature of fat bundles is that they provide
a rich source of nontrivial symplectic characteristic classes.
This follows from the functoriality of the coupling form.
More precisely, if $\Omega $ is the coupling form for a Hamiltonian
connection in a bundle 
$$
(M^{2n},\omega) \to E \stackrel{\pi}\to ~B
$$ 
then the fibre integrals 
$\mu_k:=\pi_![\Omega^{n+k}]\in H^{2k}(B\Ham\Mo)$ 
define Hamiltonian characteristic classes. Thus, if the coupling form
is nondegenerate then the fibre integral of the top power is a nonzero
top cohomology class of the base. If the base is a sphere, then the
corresponding fibre integral in the universal bundle is an
indecomposable class.  This is the first step to the understanding the
ring structure of $H^*(B\Ham\Mo))$.

Certain explicit examples of symplectically fat bundles are discussed
by Guillemin, Lerman and Sternberg in \cite{MR98d:58074,MR952456}
(mostly fibrations of coadjoint orbits over coadjoint orbits) and by
the first two authors in \cite{MR1414780}.  Some finiteness results
for fat bundles were obtained by Derdzi\'nski and Rigas in
\cite{MR610960} and Chaves in \cite{MR1286927}. Notice, however, that
they investigate quite strong fatness conditions as their motivation
is to construct metrics of positive sectional curvature.

\subsection{The main results}
\begin{enumerate}
\item
In Theorem \ref{T:weinstein}, we establish an equivalence between the
nondegeneracy of a coupling form and the existence of a fat
vector. 

\item
In Theorem \ref{T:lerman}, we describe fat vectors with respect to the
canonical invariant connection in a principal bundle of the form
$$
H\rightarrow G\rightarrow G/H.
$$ 
This has been done by Lerman in \cite{MR952456} for compact semisimple Lie
groups. A compact semisimple Lie group is a compact real form of a
complex semisimple group. We observe that Lerman's proof works also
for noncompact real forms.  This gives a generalization of his result
and has applications, going beyond the examples known from \cite{MR98d:58074}
and \cite{MR952456}. 

\item
An application of the above result is a certain duality between
fat bundles over dual symmetric spaces. This duality is discussed in
Section \ref{S:duality}.  Examples include pairs where one manifold
is complex and not K\"ahler and the second is K\"ahler.

\item
In Theorem \ref{T:pinched}, we prove that the orthonormal frame bundle
over a manifold of pinched curvature with sufficiently small pinching
constant admits fat vectors. As a consequence, we get a conceptually
simpler proof of Reznikov's construction of symplectic forms on
twistor bundles \cite{MR1225431}. In fact, the result of Reznikov is a
consequence of the existence of a fat vector.

\item
The same method yields the structure of a symplectically fat fibre
bundle on a locally homogeneous complex manifold in the sense of
Griffiths and Schmid \cite{MR0259958}, see Theorem \ref{T:lhm}. A special
case is also mentioned by Amor\'os et al. in \cite[Remark 6.18]{MR97d:32037}. 
We see that this is an exemplification of a
general construction of a fat bundle.

\item
In Theorem \ref{T:characteristic}, we show that the $\mu_n$ class is
indecomposable in the cohomology ring $H^*(B\Ham\Mo)$ of the
classifying space of the group of Hamiltonian diffeomorphisms of a
certain coadjoint orbit of $SO(2n).$

\item
In Section \ref{S:infty}, we prove that the tautological bundle
over an infinite dimensional space of symplectic configurations
(in the sense of Gal and K\k edra \cite{MR2219218})
admits fat vectors. This provides a large class of examples of
infinite dimensional symplectic manifolds.
\end{enumerate}

\begin{remark}
Fat bundles have a physical meaning.  The standard model of
elementary particles is a geometry of a certain principal bundle over
the spacetime. The coupling form was found by Sternberg as a
description of the Yang-Mills field interacting with a particle
arising from an irreducible representation of the structure group.

A reduction of the structure group in the standard model is known as
breaking of symmetry. Our Theorem \ref{T:lerman} describes, in a
sense, an opposite phenomenon. More precisely, a fat vector for the
canonical invariant connection in a bundle $H\to G\to G/H$ is an
element $X\in \mathfrak h=\mathfrak g$ avoiding certain Weyl
chambers. It provides a symplectic structure on the associated bundle
$H/V \to G/V = G\times_H H/V \to G/H$.  When $X$ approaches a
forbidden Weyl chamber this symplectic structure degenerates. This is
due to the fact that the isotropy subgroup $V\subset G$ becomes bigger
and it is no longer a subgroup of $H$.
\end{remark}

\subsection{Acknowledgements}
The second author would like to thank Eugene Lerman for answering his
questions. He is also indebted to IHES and the Max-Planck-Institute in
Bonn for hospitality during the work on this paper. The second author
is partly supported by the Ministry of Science and Higher Education,
grant no. 1P03A 03330.

\section{Coupling forms and fat vectors}\label{S:weinstein}

Let $\Mo$ be a closed symplectic $2n$-manifold and let $\Mo
\stackrel{i}\to E \stackrel{\pi}\to B$ be a Hamiltonian bundle, that
is, a bundle with structure group acting on $M$ by Hamiltonian
diffeomorphisms.  A closed two form $\Omega\in \Omega^2(E)$ is called
a {\bf {\em coupling form}} if
\begin{enumerate}
\item $i^*\Omega = \omega$ and
\item $\pi_!\Omega^{n+1}=0$, where $\dim M = 2n.$
\end{enumerate}

We construct the coupling form associated to a connection following
Guillemin-Lerman-Sternberg \cite{MR98d:58074}. Let $G\to P\to B$ be the associated
principal bundle with a connection $\theta \in \Omega^1(P,\mathfrak g)$
with Hamiltonian reduced holonomy. Let $\Theta \in \Omega^2(P,\mathfrak g)$
be the curvature two--form. The connection defines a horizontal distribution
$\mathcal H \subset TE.$  Let $\Psi:M\to \mathfrak g^*$ be the moment map
for the Hamiltonian action $G\to \Ham\Mo.$
The coupling form $\Omega$ is defined so that
the horizontal distribution $\mathcal H$ is $\Omega$-orthogonal to the
vertical distribution $\ker d\pi$ and

$$
\Omega_{[p,x]}(X,Y)=
\begin{cases}
\omega_x(X,Y) & \text{ if } X,Y \in \ker d\pi\\
\langle \Psi(x),\Theta_p(X^*,Y^*)\rangle  & \text{ if } X,Y \in \mathcal H
\end{cases}
$$
The vectors $X,Y\in T_{[p,x]}E$ are the images of $X^*,Y^*\in T_pP$
under the map $x:P \to E$ defined by a point $x\in M$ by the formula
$$
x(p):=[x,p]\in E = P\times_{G} M.
$$  
The closeness of this form is proved on page 9 in \cite{MR98d:58074} or on
page 225 in McDuff-Salamon \cite{MR97b:58062}.

Conversely, given a coupling form $\Omega\in \Omega^2(E)$  the
corresponding connection is defined by the following horizontal
distribution
$$
\mathcal H :=
\{X\in TE \,|\, \Omega(X,V)=0\,\,\text{ for all } V\in \ker d\pi\,\}.
$$
The curvature two-form on the associated principal bundle
$$
\Ham\Mo \to P \to B
$$
is given by
$$
\Theta_p(X^*,Y^*) := \Omega_{[p,-]}(X,Y)\in C^{\infty}(M),
$$ 
for any horizontal vectors $X^*,Y^*\in T_pP$.  Recall that the Lie
algebra of the group of $\Ham\Mo$ Hamiltonian diffeomorphisms is
identified with the space $C^{\infty}(M)$ of smooth functions of the
zero mean with respect to the volume form defined by the symplectic
structure.

The evaluation at a point defines an embedding $M \to C^{\infty}(M)^*$
which is the moment map for the action of $\Ham\Mo$ on $\Mo$. A
point $x\in M$ is called {\bf {\em fat}} with respect to the Hamiltonian
connection $\theta$ if the curvature two-form $\Theta $ evaluated at
$x$ is nondegenerate on the horizontal distribution.  That is
if the two-form
$$
(X^*,Y^*)\mapsto \Omega_{[p,x]}(X,Y)
$$
is nondegenerate on $\mathcal H \subset TP.$

It is now clear that the nondegeneracy of the coupling form is
equivalent to the existence of a fat point. Due to the equivariance of
the moment map, if there exist one fat point then every point in $M$
is fat as $M$ is a coadjoint orbit of $\Ham\Mo$. This proves the first
part of the following.

\begin{theorem}\label{T:weinstein}
Let $(M,\omega ) \to E \stackrel{\pi}\to B$ be a Hamiltonian bundle.
It admits a connection $\mathcal H$ with a nondegenerate
coupling form if and only if the associated connection
$\theta$  in the principal bundle
$$
\operatorname{Ham}(M,\omega)\to P \to B
$$
admits a fat point 
$x\in M\subset C^{\infty}(M)^*=\mathfrak{ham}(M,\omega)^*$.

If $H\subset \Ham\Mo$ is the connected component of the
holonomy group of the
above connection then every element $u\in \mathfrak h^*$ in
the image of the moment map $\mu:M\to \mathfrak h^*$
is fat.\qed

\end{theorem}

The second statement is implied by the following lemma whose
proof is straightforward.

\begin{lemma}
Let $H \to Q \to B$ be a reduction of a principal bundle
$G\to P \to B$ equipped with the induced connection.
A vector $u\in \mathfrak h^*$ is fat if and only if
every $v\in \mathfrak g^*$ equal to $u$ when restricted
to $\mathfrak h$ is fat.\qed
\end{lemma}

\begin{example}\label{E:su(2)}
Consider the bundle 
$\mathbb{CP}^1 \to \mathbb{CP}^3 \to \mathbb{HP}^1.$
For the standard symplectic form $\omega_3$ on $\mathbb{CP}^3$ and any
positive real number $r>0$ the form $\Omega_r:=r\cdot \omega_3$ is a
coupling form for the scaled standard symplectic form $r\cdot
\omega_1$ on $\mathbb{CP}^1$. This coupling form induces a connection
with holonomy group equal to $PSU(2)$.  Since $\Omega_r$ is
nondegenerate for all $r>0$, we get that every nonzero vector in
$\mathfrak{psu}(2)^*$ is fat with respect to the induced connection on
the principal bundle
$$
PSU(2) \to P \to \mathbb{HP}^1=S^4.
$$
We used the fact that the manifolds $(\mathbb{CP}^1,r\cdot \omega_1)$
are all nonzero coadjoint orbits of $PSU(2)$ (cf. the remark at the
bottom of page 224 of Weinstein \cite{MR82a:53038}).
\end{example}

\section{Fatness of the canonical connection in a principal bundle 
$H\to G\to G/H$}\label{S:lerman}
In the paper \cite{MR952456}, Lerman investigated the fatness
of the canonical invariant connection in a principal bundle
$$
H\rightarrow G\rightarrow G/H,
$$
where
$G$ is compact and semisimple, and $G/H$ is a coadjoint
orbit. We shall show in this section that his proof essentially
works in a more general situation. Namely, that $G$
is a semisimple Lie group and $H\subset G$ is a subgroup of maximal
rank such that the Killing form for $G$ is nondegenerate on the Lie
algebra $\mathfrak h \subset \mathfrak g$ of $H$.

Let us start with several known facts from Lie theory and introduce
notation. We denote by $\mathfrak{g}$ the Lie algebra of a Lie group
$G$. The symbol $\mathfrak{g}^c$ denotes the complexification.  Let
$\mathfrak{t}$ be a maximal abelian subalgebra in $\mathfrak{h}$. Then
$\mathfrak{t}^c$ is a Cartan subalgebra in $\mathfrak{g}^c$.  We denote by
$\Delta=\Delta(\mathfrak{g}^c,\mathfrak{t}^c)$ the root system of $\mathfrak{g}^c$
with respect to $\mathfrak{t}^c$. Under these choices
the root system for $\mathfrak{h}^c$ is a subsystem of $\Delta$. Denote
this subsystem as $\Delta(\mathfrak{h})$.

If the Killing form $B$ is nondegenerate on $\mathfrak h$ then the
subspace 
$$
\mathfrak m:= 
\{X\in \mathfrak g\,|\, B(X,Y)=0,\, \text{for all } Y\in \mathfrak h\,\}
$$
defines a decomposition
$$
\mathfrak{g}=\mathfrak{h}\oplus\mathfrak{m}. 
$$
The decomposition is $\operatorname{ad}_H$-invariant and the
restriction of the Killing form to $\mathfrak m$ is nondegenerate
(see Theorem 3.5 in Section X of \cite{MR97c:53001b}).
The decomposition complexifies to
$\mathfrak{g}^c=\mathfrak{h}^c\oplus\mathfrak{m}^c$. 
Thus, we have root decompositions:

\begin{eqnarray*}
\mathfrak{g}^c
&=&\mathfrak{t}^c+\sum_{\alpha\in\Delta}\mathfrak{g}^{\alpha},\\
\mathfrak{h}^c
&=&\mathfrak{t}^c+\sum_{\alpha\in\Delta(\mathfrak{h})}\mathfrak{g}^{\alpha},\\
\mathfrak{m}^c 
&=&\sum_{\alpha\in\Delta\setminus\Delta(\mathfrak{h})}\mathfrak{g}^{\alpha}.
\end{eqnarray*}
Since $G$ is semisimple, the Killing form $B$ defines an isomorphism
$\mathfrak g \cong \mathfrak g^*$ between the Lie algebra of $G$ and
its dual. If the Killing form is nondegenerate on $\mathfrak h$,
the composition
$$
\mathfrak h \hookrightarrow \mathfrak g \stackrel{\cong}\longrightarrow 
\mathfrak g^* \to \mathfrak h^*
$$
is an $Ad_H$-equivariant isomorphism. Let us denote this isomorphism by 
$u\mapsto X_u$. Let $C\subset\mathfrak{t}$ be the Weyl chamber
and let $C_{\alpha}$ denote its wall determined by the root $\alpha$.

\begin{theorem}\label{T:lerman} 
Let $G$ be a semisimple Lie group, and $H\subset G$ a compact subgroup
of maximal rank. Suppose that the Killing form $B$ of $G$ is
nondegenerate on the Lie algebra $\mathfrak h\subset \mathfrak g$ of
the subgroup $H.$ The following conditions are equivalent
\begin{enumerate}
\item
A vector $u\in\mathfrak{h}^*$ is fat with respect to the
the canonical invariant connection in the principal bundle
$$
H\rightarrow G\rightarrow G/H.
$$ 
\item
The vector $X_u$ does not belong to the set
$$
Ad_H(\cup_{\alpha\in\Delta\setminus\Delta(\mathfrak{h})}C_{\alpha}.)
$$
\item
The isotropy subgroup $V\subset H$ of $u\in \mathfrak h^*$ with respect to the coadjoint
action is the centralizer of a torus in $G$.
\end{enumerate}
\end{theorem}

\begin{proof}
{\em The equivalence of the first and the second condition:}
The curvature form of the canonical connection in the given
principal bundle has the form

$$
\Omega(X^*,Y^*)=-\frac{1}{2}[X,Y]_{\mathfrak{h}}, \quad X,Y\in\mathfrak{m}
$$ 
\cite[Theorem 11.1]{MR97c:53001a}.
Hence the fatness condition is expressed as the non-degeneracy of the
form

$$
(X,Y)\rightarrow B(X_u,[X,Y]_{\mathfrak{h}}). \quad (*)
$$
Recall that here the pairing is given by the Killing form. 
Since $X_u\in\mathfrak{h}$, $B(X_u,\mathfrak{m})=0$ and we get

$$
B(X_u, [X,Y]_{\mathfrak{h}})=B(X_u,[X,Y])=B([X_u,X],Y)
$$
It follows from the hypothesis that $B$ is non-degenerate on
$\mathfrak m$ and the form (*) is nondegenerate if and only if
$[X_u,X]\neq 0.$ This is equivalent to

$$
(\ker\,\operatorname{ad}_{X_u})\cap\mathfrak{m}=\{0\}.
$$
Without loss of generality we can assume that $X_u \in \mathfrak t.$
Then the last equality is, after complexification, equivalent to
the condition that 
$$
\alpha(X_u) \neq 0
$$
for all roots 
$\alpha \in \Delta \backslash \Delta(\mathfrak h)$ 
(see the root decomposition of $\mathfrak m^c$) which means
that $X_u$ does not belong to a wall $C_{\alpha}$ 
for $\alpha \in \Delta \backslash \Delta(\mathfrak h)$. 
The general case (that is  $X_u$ is not necessarily in $\mathfrak t$)
follows since $\mathfrak{h}=\cup_{h\in H}\operatorname{Ad}_h(\mathfrak{t})$. 

\bigskip
\noindent
{\em The equivalence of the first and the third condition:}
Let $u\in \mathfrak h^*$ be fat. Then its isotropy subgroup $V\subset H$ 
is connected and has the Lie algebra   
$\mathfrak v =\{X\in \mathfrak h\,|\,[X,X_u]=0\,\}.$
Since the fatness of $u$ implies 
$$(\ker \operatorname{ad}_{X_u})\cap \mathfrak m =\{0\}$$ 
(see the previous part), we get
$\mathfrak v =\{X\in \mathfrak g\,|\,[X,X_u]=0\,\}$ which means that $V$ is
the centralizer of the torus $S:= \overline {\{\exp(tX_u)\}}\subset G.$

\bigskip
Now, suppose that the isotropy subgroup $V\subset H\subset G$ is the centralizer
of a torus $S\subset G$. Let $X_u \in \mathfrak g$ be a generator of this torus.
That is $S$ is the closure of the one-parameter subgroup defined by $X_u$.
Clearly, $V$ is the coadjoint isotropy subgroup of $u:= B(X_u,-)$.
The associated symplectic form is the fibrewise symplectic
structure on the associated bundle $H/V \to G/V \to G/H.$ 
Since it is $G$-invariant, the associated connection is induced by 
the canonical invariant connection in the principal bundle $H\to G\to G/H.$
\end{proof}

\begin{cor}
Let $K$ be a compact semisimple group and let $H\subset K$
be a closed subgroup. The canonical invariant connection in the bundle
$$
H\rightarrow K\rightarrow K/H
$$ 
admits fat vectors if and only if
$\operatorname{rank}\,K=\operatorname{rank}\,H$.  
\end{cor}
\begin{proof} 
According to Theorem \ref{T:lerman}, if
$\operatorname{rank}\,K=\operatorname{rank}\,H$ then there exist fat
vectors (for example, those which lie in the interior of the Weyl
chamber). 

Let $u\in \mathfrak h^*$ be a fat vector. By perturbing $u$ slightly,
if necessary, we may assume that its isotropy subgroup is
a maximal torus $T\subset H$. The associated bundle
$$
H/T \to K \times_H H/T =K/T \to K/H
$$
admits a fibrewise symplectic structure. For cohomological reasons
the rank the torus $T$ has to be of maximal rank in $K$. This
implies that the ranks of $K$ and $H$ are equal.
\end{proof} 

\begin{remark} 
We see that the above argument shows that if
$\operatorname{rank}\,K>\operatorname{rank}\,H$, there is no chance
for fatness even for other connections.
\end{remark}

\section{Fat vectors for orthonormal frame bundles}\label{S:examples}

\subsection{Twistor bundles over spaces of constant nonzero curvature}
\label{SS:twistor}

A twistor bundle over an even dimensional Riemannian manifold $(M,g)$
is the bundle of complex structures in the tangent spaces $T_pM.$ More
precisely, it is a bundle associated with the orthonormal frame bundle
to $M$ with fibre $SO(2n)/U(n).$ It generalizes a construction
of Penrose in dimension four \cite{MR867684,MR0469146}

Reznikov proved that twistor bundles over manifolds with suitably
pinched curvature admits fibrewise symplectic forms \cite{MR1225431}.
The following proposition proves his result for manifolds of constant
non-zero curvature.

\begin{proposition}\label{P:twistor}
Let $G$ be either $SO(2n+1)$ or $SO(2n,1)$ and let
$SO(2n) \subset G$ be the obvious inclusion. Let
$J\in \mathfrak{so}(2n)$ be a matrix with
the blocks
$$
\left(
\begin{array}{rr}
0 & -1\\
1 & 0
\end{array}
\right)
$$
on the diagonal and zeros elsewhere.
Then the vector $u:=B(J,-)\in \mathfrak{so}(2n)^*$ 
is fat with respect to the canonical connection in the bundle
$$
SO(2n)\to G \to G/SO(2n).
$$
\end{proposition}

\begin{proof}
Choose a maximal torus so that its Lie algebra $\mathfrak t$ 
consist of matrices with $2\times2$-blocks of the form
$$
\left (
\begin{array}{rr}
0 & -t_i\\
t_i & 0
\end{array}
\right )
$$
on the diagonal, where $t_i\in \mathbb R$ and $i=1,\ldots,n.$
The roots for $G$ are given by
$$
t_i - t_j \text{ for } i \neq j, \quad
\pm (t_i + t_j) \text{ for } i<j  \quad \text{ and }
\pm t_i
$$
where $i,j=1,\ldots n$.  The forbidden walls are defined by the roots
$\pm t_i$.  Since $J=(1,1,\ldots,1)$ in the coordinates $t_i$, it
belongs to no forbidden wall and according to Theorem \ref{T:lerman}
the corresponding vector $u\in \mathfrak{so}(2n)^*$ is fat.
\end{proof}

\begin{cor}[Reznikov]\label{C:twistor+}
The twistor bundle over even dimensional sphere
$$
SO(2n)/U(n) \to SO(2n+1)/U(n) \to S^{2n}
$$
is symplectically fat, i.e. it admits a fibrewise symplectic
structure.\qed
\end{cor}

\begin{cor}[Reznikov]\label{C:twistor-}
Let $\Gamma \subset SO(2n,1)$ be a lattice trivially intersecting
$SO(2n).$ The twistor bundle over an even dimensional hyperbolic manifold
$$
SO(2n)/U(n) \to \Gamma\backslash SO(2n,1)/U(n) \to \Gamma\backslash SO(2n,1)/SO(2n)
$$
is symplectically fat, i.e. it admits a fibrewise symplectic
structure.
\end{cor}
\begin{proof}
It follows from Proposition \ref{P:twistor} that the associated
bundle 
$$
SO(2n)/U(n) \to SO(2n,1)/U(n) \to SO(2n,1)/SO(2n)
$$
admits an invariant fibrewise symplectic structure. Hence it
descends to a fibrewise symplectic structure after taking the
quotient by the lattice $\Gamma$.
\end{proof}

\begin{remark}
Since the orbit $SO(2n)/U(n)$ is the minimal in the hierarchy
of coadjoint orbits in the sense of \cite[page 21]{MR98d:58074}, it
follows that there are more fat vectors defining topologically
distinct coadjoint orbits.
\end{remark}

\subsection{$SO(2n)$-bundles over manifolds of pinched curvature}\label{S:pinched}
The sectional curvature $K_g$ of a Riemannian manifold $(M,g)$ is
called {\bf {\em $\varepsilon$-pinched}} if it
satisfies the following inequality
$$
1-\varepsilon \leq \left |K_g \right |\leq 1.
$$

\begin{theorem} \label{T:pinched}
The orthonormal bundle $SO(2n)\rightarrow P\rightarrow M$ 
over Riemannian manifold $M$ with $\frac{3}{2n+1}$-pinched curvature
admits fat vectors $f\in\mathfrak{so}(2n)^*$. 
\end{theorem}

\begin{proof}
Let 
$J\in \mathfrak{so}(2n)$ be a matrix with
the blocks
$$
\begin{pmatrix}
0 & -1\\
1 & 0
\end{pmatrix}
$$
on the diagonal and zeros elsewhere.
The corresponding element $f\in\mathfrak{so}(2n)^*$ in the dual space
is defined by $\left<f,A\right> := \operatorname{Tr}(A\cdot J)$. 
We shall show that $f\in\mathfrak{so}(2n)^*$ is a fat vector. 

Let $g$ be a Riemannian metric with small enough pinching constant and
let $\Theta \in \Omega^2(P,\mathfrak{so}(2n))$ denote the
corresponding curvature two-form. Recall that the curvature tensor
$R(X,Y)$ for vector fields $X,Y$ on $M$ is related to the curvature
form of the Riemannian connection by the formula
$$
R(X,Y)Z=u(\Theta(X^*,Y^*)u^{-1}(Z)),\,\text{for} X,Y,Z\in T_{\pi(u)}M
$$
where $\pi: P\rightarrow M$ is a bundle projection, and the point
$u\in P$ is an orthogonal frame $u:\mathbb{R}^n\to T_{\pi(u)}M$.  
We need to show that
$$
(X^*,Y^*)\mapsto \langle \Theta(X^*,Y^*),f\rangle
$$
is a non-degenerate 2-form on the horizontal distribution. For any
$X^*,Y^*\in \mathcal H_u$ which are the horizontal lifts of
vectors $X,Y\in T_{\pi(u)}M$ we make the following computation. 

\begin{eqnarray*}
\left<\Theta_u(X^*,Y^*),f\right>&=& \operatorname{Tr}(\Theta_u(X^*,Y^*)\cdot J)\\
&=& \operatorname{Tr}(u\cdot \Theta_u(X^*,Y^*)\cdot J\cdot u^{-1})\\
&=& \operatorname{Tr}(R_{\pi(u)}(X,Y)\cdot u \cdot J\cdot u^{-1}).\\
\end{eqnarray*}
Let $J_u := u\cdot J \cdot u^{-1}$ denote the complex structure on
$T_{\pi(u)}M$ corresponding to the frame $u.$ It follows from the
above calculation that the vector $f$ is fat if and only if the
two-from 
$$
(X^*,Y^*)\mapsto \operatorname{Tr}(R(X,Y)\cdot J_u)
$$ 
is nondegenerate.

Let $X_1, J_uX_1,...,X_n,J_uX_n$ 
be an orthonormal basis of the horizontal subspace
$\mathcal H_u \cong T_{\pi(u)}M$. Then the inequality
$$
\operatorname{Tr}(R(X_i, J_uX_i)J_u)\neq 0
$$
for all $i=1,...,n$ implies that the desired nondegeneracy. 
This trace is calculated in a usual way, taking into
consideration the orthogonality of vectors and the fact that 
$J_u^2=-\operatorname{id}$.  We have
  
\begin{eqnarray*}
|\operatorname{Tr}(R(X_i,J_uX_i)\cdot J_u)|
&=& \left|\sum_{j=1}^ng(R(X_i,J_uX_i)J_uX_j,X_j)\right |\\
&\geq& (1-\varepsilon)-\left|\sum_{j\not=i}g(R(X_i,J_u X_i)J_uX_j,X_j)\right|\\
&\geq& (1-\varepsilon)-\sum_{j\not=i}\left|g(R(X_i,J_uX_i)J_uX_j,X_j)\right|\\
&\geq& (1-\varepsilon)-(n-1) \sum_{j\not=i}\frac{2}{3}\varepsilon\\
&\geq& (1-\varepsilon)-(n-1)\frac{2}{3}\varepsilon\\
&=& 1-\frac{2n+1}{ 3}\varepsilon.
\end{eqnarray*}
In the calculation, we used the assumption of pinched sectional
curvature:
$$
|K(X_i,J_uX_i)|=|g(R(X_i,J_uX_i)J_uX_i,X_i)|\geq 1-\varepsilon,
$$
Berger's inequality 
$$
|g(R(X_i,X_j)X_k,X_l)|\leq \frac{2}{3}\varepsilon,
$$
(\cite{MR0133781}, inequality $(7)$, p. 69)
and the skew symmetric property of the curvature tensor (in the last
two arguments). Finally, taking the pinching constant,
$$
\varepsilon < \frac{3}{2n+1}
$$
we get that the vector $f$ is fat.
\end{proof}

\begin{remark}\label{R:berger}
Notice that Berger proves the above mentioned inequality for
positively curved pinched manifolds. However, his calculation
goes through almost verbatim in the negative pinching case.
\end{remark}

\begin{cor}\label{C:pinched}
Let $\xi \in \mathfrak{so}(2n)^*$ be a covector in a small
neighbourhood of $f$ and let $M_{\xi}$ denote its coadjoint orbit.
Then the associated bundle 
$M_{\xi}\to P\times _{SO(2n)}M_{\xi}\to M$ over a manifold of 
sufficiently pinched curvature admits a fibrewise symplectic form.\qed
\end{cor}

\begin{remark}
A Riemannian manifold with positive pinched curvature is known
to be homeomorphic to a sphere. However, it is not known if 
exotic spheres admit metrics of positive curvature
(see Berger \cite[Section 12.2]{MR2002701} for a survey).

In the negative curvature the situation is completely different.
There are examples of negatively curved 
closed manifolds with arbitrarily small pinching constants.
One source of examples is due to Farrell and Jones \cite{MR1002632}.
They prove that a connected sum of a hyperbolic manifold $M$
with an exotic sphere admits a pinched negative (non-constant)
curvature and it is not diffeomorphic to $M$. Another 
construction is due to Gromov and Thurston \cite{MR892185}.
They construct negatively curved pinched metrics on
certain branched coverings of hyperbolic manifolds.
\end{remark}

\begin{example}
According to the above remark and Theorem \ref{T:pinched} we get a
rich family of symplectic structures on twistor bundles over
negatively curved manifolds. Let $M$ and $M\#\Sigma$ be homeomorphic
but not diffeomorphic manifolds equipped with metrics with pinched
negative curvature as constructed by Farrell and Jones. Are the total
spaces of the associated twistor bundles diffeomorphic? And if so then
are they symplectomorphic?
\end{example}

\section{More examples}\label{S:more_examples}
\subsection{Other fat bundles over non-symplectic bases}\label{SS:quaternionic}

The twistor bundle over a sphere is an interesting example because
the base of the bundle is not symplectic. The following proposition
provides more examples if this kind. Its proof is straightforward.

\begin{proposition}\label{P:nonsymp}
Let $K$ be a compact semisimple Lie group and let $H\subset K$
be its maximal rank subgroup with finite fundamental group.
Then the bundle
$H \to K \to K/H$ admits fat vectors with respect to the 
canonical connection and the base $K/H$ is not symplectic.\qed
\end{proposition}

\subsection{Non-homogeneous examples}\label{SS:nonhomo}

Let $T\to G\to G/T$ be a principal bundle with the canonical invariant
connection $\theta$, where the torus $T$ is a subgroup of maximal
rank.  Let $(M,\omega)$ be a closed symplectic manifold endowed with a
Hamiltonian torus action $T\to \operatorname{Ham}(M,\omega)$ with the
moment map $\Psi:M\to \mathfrak t^*.$ Since the torus $T$ is an
abelian group we can add a constant to the moment map and a resulting
mapping will be equivariant. That is, we take $\Psi + a$, where 
$a\in\mathfrak t^*$ 
and define a two--form to be equal to $\omega $ on vertical vectors and
on the horizontal distribution to be given by
$$
\Omega_{a}[p,x](X,Y):=\langle \Psi(x)+a, \Theta_p(X^*.Y^*) \rangle.
$$
By Theorem \ref{T:lerman}, all vectors away the walls of the Weyl
chambers are fat.  Choosing an element $a\in \mathfrak t^*$ such that
the image of $\Psi + a$ does not intersect the walls of the Weyl
chambers we obtain a fibrewise symplectic form on the associated
bundle
$$
M\to G\times _T M \to G/T.
$$
If $G$ is non-compact we can  divide by a
suitable lattice to obtain fat bundles over locally homogeneous
spaces $\Gamma \backslash G/T$.

This construction is an application of Thurston's theorem
because choosing an element $a\in \mathfrak t^*$ is equivalent to
pulling back the symplectic form from the base $G/T$ representing
the class $a\in \mathfrak t^* \cong H^2(G/T;\mathbb R)$.

\subsection{Locally homogeneous manifolds}\label{SS:lhm}

The next proposition is a corollary of Theorem \ref{T:lerman}.
Its proof is analogous to the one of Corollary ~\ref{C:twistor-}.

\begin{proposition}\label{P:lhm}
Let $G$ be a semisimple Lie group, $K\subset G$ its maximal compact
subgroup of maximal rank and $\Gamma \subset G$ a lattice trivially
intersecting $K$.  Let $V\subset K$ be a connected subgroup that is
the centralizer of a torus in ~$G$. Then the bundle
$$
K/V \to \Gamma \backslash G/V \to \Gamma \backslash G/K
$$
is Hamiltonian and it admits a fibrewise symplectic structure. \qed
\end{proposition}

In order to construct examples from the above proposition
one needs to find a compact subgroup $V\subset G$ that is
the centralizer of a torus in $G$. A large family of
examples is provided by {\bf {\em locally homogeneous complex
manifolds}} investigated by Griffiths and Schmid in \cite{MR0259958}.
These are manifolds of the form
$\Gamma \backslash G /V$ where $G$ is a non-compact real
form of a complex semisimple group $G^c$ and $V=G\cap P$,
where $P\subset G^c$ is a parabolic subgroup.
They are indeed complex because  $G/V\subset G^c/B$ is
an open subvariety of a complex projective variety $G^c/B$, and
therefore, inherits the $G$-invariant complex structure. 

The proof of the following lemma uses standard facts from Lie algebras
which can be found for example in Chapter 6 of \cite{MR1349140}. We also
adopt the notation from this book.

\begin{lemma}\label{L:lhm}
Let $G$ be a semisimple Lie group of {\it non-compact type} which is a
real form of a complex semisimple Lie group $G^c$.  Let $P$ be a
parabolic subgroup in $G^c$ such that $V=P\cap G$ is compact.
Then $V=Z_G(S)$ is the centralizer of a torus $S\subset G$.
\end{lemma}

\begin{proof}

Let 
$\mathfrak v \subset \mathfrak k \subset \mathfrak g \subset \mathfrak g^c$
denote the Lie algebras corresponding to the groups in the statement of 
the lemma. Here $\mathfrak k\subset \mathfrak g$ is a maximal compact
subalgebra. Let $\mathfrak p \subset \mathfrak g^c$ be the parabolic
subalgebra corresponding to $P\subset G.$

Let $\Delta$ be a root system for $\mathfrak g^c$ 
and let $\Pi \subset \Delta$ be the subsystem of simple roots
such that
$$
\mathfrak k^c = \mathfrak t \oplus 
\sum_{\alpha \in [M_k]}\mathfrak g_{\alpha}
$$
for some $M_k\subset \Pi.$ This is possible due to
the fact that $\mathfrak k^c$ is reductive in $\mathfrak g^c.$
Now observe that 
$$
\mathfrak v^c = (\mathfrak p \cap \mathfrak g)^c = 
\mathfrak p \cap \mathfrak k^c.
$$
Since $\mathfrak p$ is a parabolic subalgebra we have
$$
\mathfrak p = \mathfrak t \oplus 
\sum_{\alpha \in [M_p]\cup \Delta^+}\mathfrak g_{\alpha}
$$
where $M_p \subset \Pi$ and $\Delta^+$ is the set of positive 
roots with respect to $\Pi$. 
We then get 
$$
\mathfrak v^c = \mathfrak t \oplus 
\sum_{\alpha \in [M_p\cap M_k]} \mathfrak g_{\alpha}
$$
which means that 
$\mathfrak v^c = \mathfrak z_{\mathfrak g^c}(\mathfrak a)$
is the centralizer of an Abelian subalgebra
$\mathfrak a \subset \mathfrak t.$ There exists
a vector $X\in \mathfrak a$ such that
$\mathfrak z_{\mathfrak g^c}(\mathfrak a) = 
\mathfrak z_{\mathfrak g^c}(X)$ and this vector can be
chosen to be {\em real}, that is $X\in \mathfrak a \cap \mathfrak g.$
This implies
$$
\mathfrak v = \mathfrak v^c \cap \mathfrak g
= \mathfrak z_{\mathfrak g^c}(X) \cap \mathfrak g 
= \mathfrak z_{\mathfrak g}(X)
$$
which proves that $V \subset G$ is a centralizer of the
torus $S:= \overline {\exp(tX)}.$
\end{proof}

Notice that it follows from the above proof that
$V=Z_G(S) = Z_K(S)$ where $K\subset G$ is the maximal
compact subgroup of $G$ corresponding to the subalgebra
$\mathfrak k.$ This implies that $K/V$ is a K\"ahler
manifold and the bundle
$$
K/V \to G/V \to G/K
$$
is a Hamiltonian fat bundle. Choosing an appropriate 
lattice $\Gamma \subset G$  we obtain the following result.

\begin{theorem}\label{T:lhm}
Let $G$ be a noncompact real form of a complex semisimple Lie group
$G^c$ and let $P\subset G^c$ be a parabolic subgroup such that 
$V:=P\cap G$ is compact.  Let $K\subset G$ be a maximal compact subgroup
containing $V$ and let $\Gamma \subset G$ be a cocompact lattice
trivially intersecting $K$. Then the bundle
$$
K/V \to \Gamma \backslash G/V \to \Gamma\backslash G/K
$$
is Hamiltonian and admits a fibrewise symplectic structure. \qed
\end{theorem}

\subsection{A relation to a question of Weinstein}
\label{SS:weinstein}
Weinstein \cite{MR82a:53038} was interested in
constructing a simply connected symplectic and not K\"ahler manifold
as a total space of a bundle.  We don't know if there are simply
connected symplectically fat bundles. However, the above theorem
provides many examples of symplectically fat bundles which have
non-K\"ahler fundamental groups. Hence they are not homotopy
equivalent to K\"ahler manifolds. This idea was used first by
Reznikov in \cite{MR1225431} for twistor bundles over 
spaces of negative curvature.

\begin{proposition}\label{P:weinstein}
Let $K/V \to \Gamma \backslash G/V \to \Gamma\backslash G/K$ be a
fibre bundle as in Theorem \ref{T:lhm}.  Suppose that $G/K$ is not
Hermitian symmetric.  Then the bundle is symplectically fat
and its total space is not homotopy equivalent to a K\"ahler manifold.
\end{proposition}

\begin{proof} In view of Theorem \ref{T:lhm}, it remains to prove that 
$\Gamma\backslash G/V$ is non-K\"ahler. The proof is analogous to the
proof of Theorem 6.17 in \cite{MR97d:32037}. Suppose that the total space is
homotopy equivalent to a K\"ahler manifold $M$.  Since the bundle is
Hamiltonian with compact structure group the projection $\pi$
induces a surjective homomorphism on homology. This follows from a
general cohomological splitting for Hamiltonian bundles 
\cite[Corollary 4.10]{MR1941438}.

Now, according to a theorem of Siu \cite[Theorem 6.14]{MR97d:32037}, the
composition of the homotopy equivalence $M\to \Gamma\backslash G/V$
and the bundle projection is homotopic to a holomorphic map for some
invariant complex structure on $G/K$. This contradicts the assumption
that $G/K$ is not Hermitian symmetric.
\end{proof}

\section{Infinite dimensional examples}\label{S:infty}

Let $(M,\omega)$ and $(W,\omega_W)$ be symplectic manifolds and let
$\operatorname{Symp}(M,W)$ denote the space of symplectic embeddings
of $M$ into $W$. The group of symplectic diffeomorphisms of
$(M,\omega)$ acts freely from the right on this space of embeddings
while symplectic diffeomorphisms of $(W,\omega_W)$ act from the
right. We denote the quotient by $\operatorname{Conf}(M,W)$ and
following Gal and K\k edra \cite{MR2219218} it is called the space of
{\bf {\em symplectic configurations of $M$ in $W$}}.  There is a principal bundle
$$
\Symp\Mo  \to \Symp(M,W)\to \operatorname{Conf}(M,W).
$$
It admits a symplectic connection $\mathcal H$ whose
curvature two-form is given by the following formula
\cite[Section 4.2]{MR2219218}
$$
\Theta(X,Y)(f) := \{H_X,H_Y\}\circ f,
$$
where $f\in \Symp(M,W)$ is a symplectic embedding
and $X,Y$ are horizontal vectors at $f$. Recall, that
a vector tangent to the space of embedding is a section
of the pull back bundle $f^*(TW).$ Such a section can be
extended to a Hamiltonian vector field in a neighbourhood of
$f(M)\subset W.$ The functions $H_X,H_Y:W \to \mathbb R$ are
the Hamiltonians of these extensions.

\begin{proposition}\label{T:infty}
If the dimension of $W$ is bigger than the dimension of $M$
then the  connection $\mathcal H$ is fat.
\end{proposition}
\begin{proof}
This form is nondegenerate on the horizontal subspace at an
embedding $f$ if and only if for any nonzero function $H:W\to \mathbb
R$ constant on $f(M)$ there exists a function $F:W\to \mathbb R$ also
constant on $f(M)$ such that the Poisson bracket $\{H,F\}$ is nonzero
on $f(M)$. The statement easily follows from the local expression of
the Poisson bracket.
\end{proof}

\begin{cor}
The coupling form on the total space 
of the associated tautological bundle over symplectic configurations 
$$
M\to E:=\Symp(M,W)\times_{\Symp\Mo}M \to \operatorname{Conf}(M,W)
$$
is nondegenerate. \qed
\end{cor}

We obtain this way a large family of examples of infinite dimensional
symplectic manifolds. Notice that the symplectic form can be
calculated using the following formula
$$
\Omega_{[f,x]}\left (\frac{d}{dt}[\varphi_t,x],\frac{d}{dt}[\psi_t,x]\right )=
(\omega_W)_{f(x)}\left (\frac{d}{dt}\varphi_t(x)),\frac{d}{dt}\psi_t(x)\right )\\
$$ 
where $\varphi_0=\psi_0=f$ and $\varphi_t=\psi_t\in \Symp(M,W)$.
These examples are interesting because they admit a Hamiltonian action
of the group of Hamiltonian diffeomorphisms of the manifold $\Wo$.

\begin{proposition}\label{P:moment}
The group $\Ham\Wo$ acts on the total space $E$ of the tautological
bundle over the symplectic configurations from the left and the action
preserves the symplectic form.  Moreover, the moment map
$$
\Psi:\Symp(M,W)\times_{\Symp\Mo}M \to \mathfrak{ham}\Wo^* 
= C^{\infty}(W)^*
$$
is defined by the following formula
$$
\left<\Psi[f,x],F\right>:= F(f(x)).
$$
\end{proposition}

\begin{proof}
Let $\Phi \in \Ham\Wo$ and let $\widehat \Phi: E \to E$ be defined
by $\widehat \Phi [f,x] := [\Phi \circ f,x].$ We need to show that
$\widehat \Phi $ preserves the symplectic form $\Omega$.
This is the following calculation.
\begin{eqnarray*}
\widehat \Phi^*\Omega_{[f,x]}
\left (\frac{d}{dt}[\varphi_t,x],\frac{d}{dt}[\psi_t,x]\right)
&=&
\Omega_{[\Phi\circ f,x]}
\left (\frac{d}{dt}[\Phi\circ\varphi_t,x],\frac{d}{dt}[\Phi\circ\psi_t,x]\right )\\
&=&
(\omega_W)_{\Phi(f(x))}
\left (\frac{d}{dt}\Phi(\varphi_t(x)),\frac{d}{dt}\Phi(\psi_t(x))\right )\\
&=&
(\Phi^*\omega_W)_{f(x)}
\left (\frac{d}{dt}\varphi_t(x),\frac{d}{dt}\psi_t(x))\right )\\
&=&
(\omega_W)_{f(x)}
\left (\frac{d}{dt}\varphi_t(x),\frac{d}{dt}\psi_t(x))\right )\\
&=&
\Omega_{[f,x]}
\left (\frac{d}{dt}[\varphi_t,x],\frac{d}{dt}[\psi_t,x]\right)
\end{eqnarray*}

Let $ev:E \to W$ be the evaluation map defined by $ev[f,x]=f(x)$.
Given a function $F:W\to \mathbb R$ we denote by $\underline F$
the vector field on $E$ generated by $F.$ We have to check that 
$$
d(F\circ ev) = i_{\underline F}\Omega
$$
for every function $F:W\to\mathbb R$. Let $f_t \in \Symp(M,W)$ be
a path of embeddings with $f_0=f$ and let $X_F$ denote the
Hamiltonian vector field on $\Wo$ generated by the flow
$\Phi_t\in \Ham\Wo$ defined by the Hamiltonian $F$.
We have the following computation.

\begin{eqnarray*}
d(F\circ ev)_{[f,x]}\left (\frac{d}{dt}[f_t,x]\right )&=&
\frac{d}{dt}F(ev[f_t,x])\\
&=&
\frac{d}{dt}F(f_t(x))\\
&=&dF(f_t(x))\\
&=&
(i_{X_F}\omega_W)_{f(x)}\left (\frac{d}{dt}f_t(x)\right )\\
&=&
(\omega_W)_{f(x)}\left (\frac{d}{dt}\Phi_t(f(x)),\frac{d}{dt}f_t(x)\right )\\
&=&
\Omega_{[f,x]}\left (\frac{d}{dt}[\Phi_t\circ f,x],\frac{d}{dt}[f_t,x]\right )\\
&=&
(i_{\underline F}\Omega)\left (\frac{d}{dt}[f_t,x]\right )
\end{eqnarray*}

\end{proof}

\section{A duality of fat bundles} \label{S:duality}

Let $G$ be a noncompact semisimple Lie group with a maximal compact
subgroup $K\subset G$ and let $\Gamma \subset G$ be an irreducible
cocompact lattice trivially intersecting $K$. Assume that $G$ is a
real from of a complex semisimple group $G^c$.  Let $M\subset G^c$ be
a maximal compact subgroup. It was observed by Okun \cite{MR1875614}
that the map $\Gamma \backslash G/K \to BK$ classifying the principal
bundle
$$
K\to \Gamma\backslash G \to \Gamma \backslash G/K = B\Gamma
$$
lifts to a map $\beta\colon \Gamma \backslash G/K \to M/K$ after passing to a
sublattice of finite index if necessary. Moreover this map is
tangential, that is the pull back of the tangent bundle of the target
manifold is isomorphic to the tangent bundle of the source
manifold. The homomorphism $H^*(M/K)\to H^*(B\Gamma)$ induced
by the above tangential map was investigated by Matsushima
in \cite{MR0141138}. His main result says that this homomorphism
is injective in all degrees and surjective in degrees
smaller than the rank of $G.$

We want to apply this observation to a special case of
those semisimple noncompact Lie groups $G$ whose
maximal compact subgroup $K\subset G$ is of maximal rank.
We obtain the following duality of fat bundles.

\begin{proposition}\label{P:dual}
Let $\Gamma \backslash G/K$ be a locally symmetric space of noncompact type
and let $M/K$ be its dual. Assume moreover that $K\subset G$ is a maximal
compact subgroup of maximal rank. The following statements hold;
\begin{enumerate}
\item
There is a pullback diagram of $K$-principal bundles
$$
\xymatrix
{
K\ar[r] \ar[d] & K \ar[d]\\
\Gamma\backslash G \ar[r]^{\tilde \beta} \ar[d]^p & M  \ar[d]^{\pi}\\
B\Gamma \ar[r]^{\beta} & M/K 
}
$$ 
\item
Both bundles have the same nonempty sets of fat vectors. 
\item
The pull back of the horizontal distribution
$\mathcal H_{M}\subset TM$ is isomorphic (as a bundle) 
to the horizontal distribution
$\mathcal H_{\Gamma \backslash G}\subset T(\Gamma \backslash G).$
\item 
The morphism $\tilde \beta$ does not preserve the connections.
\end{enumerate}
\end{proposition}

\begin{proof}\hfill

\begin{enumerate}
\item
The existence of the pull back diagram is a direct application of
the above mentioned result of Okun. 
\item
It follows from Theorem \ref{T:lerman} that the bundles have the same
and nonempty set of fat vectors.
\item
Since the Okun map $\beta $ is tangential we get
$\beta^*(T(M/K)) \cong T(\Gamma \backslash G/K)$. On the other
hand we have the following isomorphisms of bundles
$\pi^*(T(M/K)) \cong \mathcal H_{M}$ and
$p^*(T(\Gamma \backslash G/K)) \cong \mathcal H_{\Gamma\backslash G}.$ 
Composing these isomorphisms we get the statement.
\item
Since $M/K$ is simply connected and $B\Gamma$ is not, the Okun
map has singularities. Hence the bundle morphism cannot
preserve the connections.
\end{enumerate}
\end{proof}

The following corollary is straighforward. The third part follows from
the uniqueness of the coupling class (see Section
\ref{SS:coupling_class}).

\begin{cor}\label{C:dual}
Let $\xi \in \mathfrak k^*$ be a fat vector for the connections
in Proposition \ref{P:dual} and let $H\subset K$ denote its
isotropy subgroup. Then the following statements hold; 
\begin{enumerate}
\item There is a pull back diagram of the
associated bundles.
$$
\xymatrix
{
K/H \ar[r] \ar[d] & K/H\ar[d]\\
\Gamma\backslash G/H \ar[r]^{\widehat \beta} \ar[d]^p & M/H  \ar[d]^{\pi}\\
B\Gamma \ar[r]^{\beta} & M/K 
}
$$
\item 
The map $\widehat \beta$ is tangential.
\item
The map $\widehat \beta$ preserves the cohomology classes
of symplectic forms. In other words, it is a c-symplectic morphism.
\end{enumerate}\qed
\end{cor}

\begin{example}
Applying the above proposition to the twistor bundle over 
a hyperbolic manifold $X:=\Gamma \backslash SO(2n,1)/SO(2n)$
we obtain the following pull back diagram.
$$
\xymatrix
{
SO(2n)/U(n) \ar[r] \ar[d] & SO(2n)/U(n)\ar[d]\\
\Gamma\backslash SO(2n,1)/U(n) \ar[r] \ar[d] & SO(2n+1)/U(n)  \ar[d]\\
X \ar[r] & S^{2n}
}
$$

The lattice $\Gamma$ is a non-K\"ahler group (see Theorem 6.22 in
\cite{MR97d:32037}) and hence the symplectic manifold 
$\Gamma \backslash SO(2n,1)/SO(2n)$ 
is not K\"ahler. On the other hand the
manifold $SO(2n+1)/U(n)$ is K\"ahler.
\end{example}

\begin{example}
Let $G$ be a noncompact real form of a semisimple complex
Lie group $G^c$. Let $P\subset G^c$ be a parabolic subgroup
such that  $H=G\cap P$ is compact. Let $\Gamma \subset G$
be a suitable cocompact and irreducible lattice. We obtain 
the pair of dual manifolds
$$
\Gamma \backslash G/H \text{ and } M/H
$$
and, according to Griffiths and Schmid the first one
is complex. In general it is not K\"ahler.  The second
manifold is always K\"ahler.  
\end{example}

\section{Applications to Hamiltonian characteristic classes}
\label{S:applications}

\subsection{Hamiltonian characteristic classes defined by the
coupling class}\label{SS:coupling_class}

The cohomology class of a coupling form is called the 
{\bf {\em coupling class}}. 
If either the fibre or the base is simply connected then the coupling
class is unique. This follows directly from the two conditions
defining the coupling form and the Leray-Serre spectral sequence for
the fibration.  Since the classifying space of the group of
Hamiltonian diffeomorphisms is simply connected, the coupling class
can be defined universally as the cohomology class $\Omega \in
H^2(M_{\Ham})$ where
$$
\Mo\to M_{\Ham}\stackrel{\pi}\to B\Ham\Mo
$$
is the universal Hamiltonian fibration for $\Mo$.
By integrating the powers of the coupling class
we get Hamiltonian characteristic classes
$\mu_k:=\pi_!(\Omega^{n+k})\in H^{2k}(B\Ham\Mo).$

Let $\Mo\to E \stackrel{p}\to B$ be a Hamiltonian bundle classified by
a map $c:B\to B\Ham\Mo.$ The uniqueness of the coupling class and the
functoriality of the fibre integration implies that the characteristic
class $\mu_k(E):=c^*(\mu_k)$ is equal to the fibre integral
$p_!(\Omega_E^{n+k})$, where $\Omega_E$ is the coupling class of the
bundle $E.$ The following proposition is straightforward (cf. Theorem 4.1
in Weinstein \cite{MR82a:53038}).

\begin{proposition}\label{P:characteristic}
Let $(M^{2n},\omega)\to E\to B$ be a Hamiltonian bundle over
$2k$-dimensional base.  If the coupling form $\Omega_E$ is
nondegenerate then the characteristic class $\mu_k(E)$ is nonzero.
\qed
\end{proposition}

\subsection{Fat bundles over spheres}\label{SS:spheres}
Applying the last proposition to the examples of
fat bundles over spheres we obtain the following
result.

\begin{theorem}\label{T:characteristic}
Let $SO(2n)\to P\to S^{2n}$ be the frame bundle. Let $\xi \in
\mathfrak{so}(2n)^*$ be a fat vector with respect to the canonical
connection (see Section \ref{SS:twistor}) and let $M_{\xi}$ denote its
coadjoint orbit.  Then the class $\mu_n\in H^*(B\Ham(M_{\xi}))$ is a
nonzero indecomposable element. \qed
\end{theorem}

Indecomposability means that the class is not a sum of products of
classed of positive degree. Indeed, if the $\mu_k$ class was a sum of
products its pull back evaluated over a sphere would be zero
and this would contradict with the results of Section
\ref{SS:twistor}. In other words, the $\mu_k$ can be chosen to
be a generator of the cohomology ring of the classifying space of
the group of Hamiltonian diffeomorphisms of $M_{\xi}.$

\subsection{Hamiltonian actions of $SU(2)$}
The bundle $\mathbb{CP}^1\to \mathbb{CP}^2 \to \mathbb{HP}^1=S^4$
admits a fibrewise symplectic structure for every symplectic form on
the fibre. Hence every nonzero vector in $\mathfrak{su}(2)$ is fat
with respect to the induced connection in the associated principal
bundle $SU(2) \to P \to S^4$, see Example \ref{E:su(2)}.
This proposition has been first proved by Reznikov in \cite{MR2000f:53116} 
and its generalization by K\k edra and McDuff in \cite{MR2115670}.

\begin{proposition}\label{P:SU(2)}
Let $SU(2)\to \Ham\Mo$ be a nontrivial Hamiltonian action. It induces
a surjective homomorphism
$$
H^4(B\Ham\Mo;\mathbb R)\to H^4(BSU(2);\mathbb R)
$$
\end{proposition}

\begin{proof}
Consider the associated bundle
$$
\Mo \to E:=P\times_{SU(2)}M \to S^4
$$
where $P$ is the above principal bundle endowed with a connection
with respect to which all nonzero vectors are fat. Let
$\mu:M\to \mathfrak{su}(2)^*$ be the moment map.
The coupling form for the above fat connection is nondegenerate
away the subset
$$
P\times _{SU(2)}\mu^{-1}(0) \subset E.
$$
Since the complement of this subset is connected and of full measure,
we get that the integral of the top power of the coupling form over
$E$ is nonzero. Hence the fibre integral of the top power is nontrivial
in $H^4(S^4;\mathbb R)$ which proves the statement.
\end{proof}

\bibliography{../bib/bibliography}
\bibliographystyle{plain}

\end{document}